\documentclass[preprint,12pt]{elsarticle}
\pdfoutput=1

\usepackage{lineno,hyperref}
\modulolinenumbers[5]

\usepackage{ifpdf}
\usepackage{bm}
\usepackage{mathdots}
\usepackage{soul, color}
\usepackage{amsmath}
\usepackage{mathrsfs}
\usepackage{graphicx}
\usepackage{epstopdf}
\usepackage{epsfig}

\newenvironment{bsmallmatrix}
  {\left[\begin{smallmatrix}}
  {\end{smallmatrix}\right]}

\newtheorem{definition}{Definition}[section]

\newtheorem{theorem}{Theorem}[section]
\newtheorem{lemma}[theorem]{Lemma}

\newtheorem{corollary}[theorem]{Corollary}

\newenvironment{proof}[1][Proof]{\begin{trivlist}
\item[\hskip \labelsep {\bfseries #1}]}{\end{trivlist}}

\newenvironment{remark}[1][Remark]{\begin{trivlist}
\item[\hskip \labelsep {\bfseries #1}]}{\end{trivlist}}

\usepackage{amssymb}

\journal{Journal}

\begin{document}

\begin{frontmatter}



\title{Circulant preconditioners for functions of Hermitian Toeplitz matrices}

\author{Sean Hon\fnref{label1}}

\address{Mathematical Institute, University of Oxford, Radcliffe Observatory Quarter, Oxford, OX2 6GG, United Kingdom}
\ead{hon@maths.ox.ac.uk}
\fntext[label1]{The research of the author was partially supported by the Croucher Foundation of Hong Kong.}

\begin{abstract}
Circulant preconditioners for functions of matrices have been recently of interest. In particular, several authors proposed the use of the optimal circulant preconditioners as well as the superoptimal circulant preconditioners in this context and numerically illustrated that such preconditioners are effective for certain functions of Toeplitz matrices. Motivated by their results, we propose in this work the absolute value superoptimal circulant preconditioners and provide several theorems that analytically show the effectiveness of such circulant preconditioners for systems defined by functions of Toeplitz matrices. Namely, we show that the eigenvalues of the preconditioned matrices are clustered around $\pm 1$ and rapid convergence of Krylov subspace methods can therefore be expected. Moreover, we show that our results can be extended to functions of block Toeplitz matrices with Toeplitz blocks provided that the optimal block circulant matrices with circulant blocks are used as preconditioners. Numerical examples are given to support our theoretical results.
\end{abstract}


\end{frontmatter}


\section{Introduction}

Circulant preconditioners for functions of matrices have been used recently. Jin, Zhao, and Tam \cite{Jin2014224} proposed using the optimal circulant preconditioners in this context. Later, Bai, Jin, and Yao \cite{bai2015} also suggested the use of the superoptimal circulant preconditioners for the same problem. The authors provided several properties of functions of circulant matrices and then numerically demonstrated the effectiveness of their proposed preconditioners for certain functions of Toeplitz matrices. Note that functions of Toeplitz matrices have some crucial applications, for example in option pricing \cite{duffy2013finite, SACHS20081687} where the Toeplitz matrix exponentials arise.

Motivated by the authors' results, we first propose the use of absolute value superoptimal circulant preconditioners and provide several theoretical results that account for the success of such preconditioners for functions of Toeplitz matrices. We then show that the optimal circulant preconditioners are also effective for functions of Toeplitz matrices in the block matrix case. In other words, we provide in this work two special kinds of matrices for which the optimal type circulant preconditioners are successful, which broaden the use of such preconditioners in preconditioning for functions of matrices addressed by the authors in \cite{Jin2014224,bai2015}.

In our main results, we show that $|h(T_n)|^{-1}h(A_n)$ can be decomposed into the sum of a unitary matrix, a low rank matrix, and a small norm matrix for sufficiently large $n$, where $h(z)$ is an analytic function, $A_n$ is the Toeplitz matrix generated by a positive function in the Wiener class, and $T_n$ is the superoptimal circulant preconditioner derived from $A_n$. The clustered spectra of $|h(T_n)|^{-1}h(A_n)$ around $\pm 1$ can then be shown. As $|h(T_n)|$ is Hermitian positive definite by definition, Krylov subspace methods such as the minimal residual (MINRES) method can be employed for rapid convergence (see for example \cite{ANU:9672992,MR3235759}). When $T_n$ is replaced by Strang's circulant preconditioner \cite{Strang:1986:PTM:12330.12335} $S_n$, we show that similar results on $|h(S_n)|^{-1}h(A_n)$ also hold.



Moreover, considering the optimal circulant preconditioners, we further extend our previous results \cite{Hon2018148,Hon2018} to the block matrix case. Namely, we propose the absolute value optimal block circulant preconditioners with circulant blocks for functions of block Toeplitz matrices with Toeplitz blocks. Several theorems that show their spectra around $\pm 1$ are also given.

We remark that functions of Toeplitz matrices are not Toeplitz matrices in general (as a simple example consider $A_n^2$). The standard preconditioning techniques for Toeplitz systems do not straightforwardly apply to $h(A_n)$. However, when $h(z)=z$, the standard results provided in \cite{MR1098422} on the superoptimal circulant preconditioners for the usual Toeplitz systems are recovered. The same arguments also hold for the block Toeplitz matrix case.

By the diagonalisation of circulant matrices $C_n=U_n^{*}\Lambda_n U_n$, where $U_n\in \mathbb{C}^{n\times n}$ is the Fourier matrix of which the entries are given by $[U_n]_{jk}=\frac{1}{\sqrt{n}}e^{-2\pi \mathbf{i} j k/n}$ with $j,k=0,1,\dots,n-1$, we have $ |h(C_n)| = U_n^{*} |h(\Lambda_n)|U_n.$ In other words, $|h(C_n)|$ is also a circulant matrix. Therefore, for any vector $\mathbf{d}$ the product $|h(C_n)|^{-1}\mathbf{d}$ can by efficiently computed by several Fast Fourier Transforms (FFTs) in $\mathcal{O}(n\log{n})$ operations. 

It must be noted that fast matrix vector multiplication with $h(A_n)$ is not readily archived. However, for $e^{A_n}$ the matrix vector multiplication can be computed efficiently in $\mathcal{O}(n\log{n})$ operations for example in \cite{doi:10.1137/090758064}.

Numerical results by the conjugate gradient (CG) method,  MINRES, and the generalised minimal residual (GMRES) method are given to support our theoretical results and to demonstrate the clusters of eigenvalues around $\pm 1$ for the preconditioned matrices.

\section{Preliminary results on Toeplitz matrices}

In this section, we first present some preliminary results on $A_n$ that will be used in the next section. 

Assuming the given Toeplitz matrix $A_n$ is associated with the function $f$ via its Fourier series defined on $[-\pi,\pi]$, we have
$$ A_n=\begin{bmatrix}{}
a_0 & a_{-1} & \cdots & a_{-n+2} & a_{-n+1} \\
a_1 & a_0 & a_{-1}   &  & a_{-n+2} \\
\vdots & a_1 & a_0 & \ddots & \vdots \\
a_{n-2} &  & \ddots & \ddots & a_{-1} \\
a_{n-1} & a_{n-2} &\cdots & a_1 & a_0
\end{bmatrix} \in \mathbb{C}^{n\times n},$$ where
$$ a_{k}=\frac{1}{2\pi} \int_{-\pi} ^{\pi}f(x) e^{-\mathbf{i} k x } \,dx,\quad k=0,\pm1,\pm2,\dots,$$ are the Fourier coefficients of $f$. The function $f$ is called the \emph{generating function} of the Toeplitz matrix. We refer to \cite{MR2108963,MR2376196,Chan:1996:CGM:240441.240445} for more discussions on other aspects of Toeplitz matrices. 

Throughout this work, we assume that $f$ is a positive function in the Wiener class, namely $$\sum_{k=-\infty}^{\infty}|a_k|<\infty.$$ Thus, the corresponding matrix $A_n$ is Hermitian positive definite for all $n$.

We now introduce the optimal type circulant preconditioners in the following. Let
$
\mathcal{M}_{U_n}=\{ U_n^*\Lambda_n U_n~|~\Lambda_n\in \mathbb{C}^{n\times n}~\text{is any $n\times n$ diagonal matrix}\}
$
be the set of all circulant matrices.

The superoptimal preconditioner $T_n \in \mathbb{C}^{n\times n}$ proposed by Tyrtyshnikov \cite{doi:10.1137/0613030} for $ A_n$ is defined to be $$\min_{C_n \in \mathcal{M}_{U_n}} \|I_n-C_n^{-1}A_n\|_{F},$$ where $\|\cdot\|_F$ is the Frobenius norm. 

Similarly, the optimal circulant preconditioner ${c}(A_n) \in \mathbb{C}^{n\times n}$ by T. Chan \cite{doi:10.1137/0909051} for $ A_n $ is defined to be $$\min_{C_n \in \mathcal{M}_{U_n}} \|C_n-A_n\|_{F}.$$

We also provide the following definitions in relation to clustered spectra around $\pm 1$ and absolute value circulant preconditioners.

\begin{definition}\cite[Definition 4.5]{doi:10.1137/140974213}\label{def:clsutered_spec}
A sequence of matrices $\{H_n\}_{n=1}^{\infty}$ is said to have clustered spectrum around $\pm 1$ if for any $\epsilon>0$ there exist positive integers $M$ and $N$ such that for all $n>N$, at most $M$ eigenvalues $\lambda$ of $H_n$ are such that $|\lambda-1|>\epsilon$ and $|\lambda+1|>\epsilon$.
\end{definition}

\begin{definition}\cite{doi:10.1137/140974213}\label{def:abs_C}
Let $C_n \in \mathbb{C}^{n\times n}$ be a circulant matrix. The \emph{absolute value circulant matrix} $|C_n| \in \mathbb{C}^{n\times n}$ for $C_n$ is defined by
\begin{eqnarray}\nonumber
|C_n| &=& (C_n^* C_n)^{1/2} \\\nonumber
&=& (C_n C_n^*)^{1/2}\\\nonumber
&=&U_n^*|\Omega_n|U_n,
\end{eqnarray}
where $U_n \in \mathbb{C}^{n \times n}$ is a Fourier matrix and $|\Omega_n| \in \mathbb{R}^{n \times n}$ is the diagonal matrix in the eigendecomposition of $C_n$ with all entries replaced by their magnitude.
\end{definition} 

\begin{remark}
Note that $|C_n|$ is Hermitian positive definite by definition provided that $C_n$ is nonsingular.
\end{remark}

Given an analytic function $h$, $|h(C_n)|=U_n^*|h(\Omega_n)|U_n$ is well-defined and is also a circulant matrix by Definition \ref{def:abs_C}.

We will require the following lemma for showing Theorems \ref{lem:mianTAdecompose} and \ref{lem:mianSAdecompose} concerning the matrix decomposition involving a Toeplitz matrix and its corresponding superoptimal/Strang's circulant preconditioner.

%

%
%

\begin{lemma}\cite{MR1098422, doi:10.1137/0610039}
\label{lem:TCBound}
Let $f$ be a positive function in the Wiener class. Let $A_n \in \mathbb{C}^{n\times n} $ be the Toeplitz matrix generated by $f$, $T_n \in \mathbb{C}^{n\times n}$ be the superoptimal circulant preconditioner for $A_n$, and $S_n \in \mathbb{C}^{n \times n}$ be Strang's circulant preconditioner for $A_n$. Then
\[
\|A_n \|_2 \leq f_{\max},\quad \|T_{n} \|_2 \leq \frac{f_{\max}^2}{f_{\min}},\quad \text{and}\quad \|S_n \|_2 \leq f_{\max},
\quad n=1,2,\dots.
\]
\end{lemma}

%
%

%

The following theorem is in fact a restatement of Theorem 5 in \cite{MR1098422}, which will be used for showing our main theorem in the next section.

\begin{theorem}\cite[Theorem~5]{MR1098422}
\label{lem:mianTAdecompose}
Let $f$ be a positive function in the Wiener class. Let $A_n \in \mathbb{C}^{n \times n}$ be the Toeplitz matrix generated by $f$ and $T_n \in \mathbb{C}^{n \times n}$ be the superoptimal circulant preconditioner for $A_n$. Then for all $\epsilon>0$ there exist integers $N$ and $M>0$ such that for all $n>N$
\[T_n  - A_n  = V_n + W_n,\]where
\[ \text{rank}(V_n) \leq 2M \quad \text{and} \quad \|W_n\|_2 \leq \epsilon. \]
\end{theorem}
%
%

Similar to the superoptimal circulant preconditioners, the following theorem are needed to show our results on Strang's circulant preconditioners for functions of Toeplitz matrices.

%
%

\begin{theorem}\cite[Theorem~2]{doi:10.1137/0610039}
\label{lem:mianSAdecompose}
Let $f$ be a positive function in the Wiener class. Let $A_n \in \mathbb{C}^{n \times n}$ be the Toeplitz matrix generated by $f$ and $S_n \in \mathbb{C}^{n \times n}$ be Strang's circulant preconditioner for $A_n$. Then for all $\epsilon>0$ there exist positive integers $N$ and $M$ such that for all $n>N$
\[S_n  - A_n  = V_n + W_n,\]where
\[ \text{rank}(V_n) \leq 2M \quad \text{and} \quad\|W_n\|_2 \leq \epsilon. \]
\end{theorem}

\section{Main results}
In this section, we show that the preconditioned matrix $|h(T_n )|^{-1}h(A_n )$ can be decomposed into the sum of a unitary matrix, a low rank matrix, and a small norm matrix for sufficiently large $n$ under certain conditions. As a result, the spectra of $|h(T_n )|^{-1}h(A_n )$ are clustered around $\pm 1$.

We first provide two theorems concerning functions of matrices that will be useful in our analysis.


%
%
%
%
%
%

\begin{theorem}\cite[Theorem~4.7]{MR2396439}
\label{lem:matrixtaylorseries}
Suppose $h$ has a Taylor series expansion \[ h(z)=\sum_{k=0}^{\infty}a_k(z-\alpha)^k, \] where $a_k=\frac{h^{(k)}(\alpha)}{k!}$, with the radius of convergence $r$. If $A_n\in \mathbb{C}^{n\times n}$, then $h(A_n)$ is defined and is given by \[h(A_n) = \sum_{k=0}^{\infty}a_k(A_n-\alpha I_n)^{k}\] if and only if the distinct eigenvalues $\lambda_1$, $\cdots$, $\lambda_s$ of $A_n$ satisfy one of the conditions

(a) $|\lambda_i-\alpha|<r$,
\vspace{1mm}

(b) $|\lambda_i-\alpha|=r$ and the series for $h^{(n_i-1)}(\lambda)$, where $n_i$ is the index of $\lambda_i$, is convergent at the point $\lambda=\lambda_i$, $i=1,\dots,s$.
\end{theorem}

\begin{theorem}\cite[Theorem~4.8]{MR2396439}
\label{lem:matrixtaylorserieserrorbd}
Suppose $h$ has a Taylor series expansion \[ h(z)=\sum_{k=0}^{\infty}a_k(z-\alpha)^k, \] where $a_k=\frac{h^{(k)}(\alpha)}{k!}$, with the radius of convergence $r$. If $A_n\in \mathbb{C}^{n\times n}$ with $\rho(A_n-\alpha I_n)<r$ then for any matrix norm $\|\cdot\|$ \[\|h(A_n) - \sum_{k=0}^{K-1}a_k(A_n-\alpha I_n)^{k}\| \leq \frac{1}{K!}\max_{0\leq t \leq 1} {\|(A_n-\alpha I_n)^K h^{(K)}(\alpha I_n + t (A_n-\alpha I_n)) \|}. \]
\end{theorem}

We now show our main results on $|h(T_n)|^{-1}h(A_n) $. Without loss of generality, we assume that $h(z)$ has the Taylor series representation
$$
h(z)=\sum_{k=0}^{\infty}a_k z^k.
$$


\begin{theorem}
\label{thm:mainTAdecomp_hon17}
Suppose $h(z)$ is an analytic function defined on $|z|<r$ with the radius of converge $r$. Let $f$ be a positive function in the Wiener class such that $\frac{f_{\max}^2}{f_{\min}}<r$. Let $A_n\in \mathbb{C}^{n \times n}$ be the Toeplitz matrix generated by $f$ and $T_n\in \mathbb{C}^{n \times n}$ be the superoptimal circulant preconditioner for $A_n$. Then for all $\epsilon > 0$ there exist positive integers $N$ and $M$ such that for all $n > N$ \[ h({T_n}) - h({A_n}) = R_n + E_n,\] where \[ \text{rank}( R_n) \leq 2M \quad \text{and}\quad \|E_n \|_2 \leq \epsilon.\]
\end{theorem}

\begin{proof}

Since $h(z)$ is analytic on $|z|<r$, it has the following Taylor series representation: $ h(z)=\sum_{k=0}^{\infty}a_k z^k $ with the radius of convergence $r=(\lim_{k\to\infty}|\frac{a_{k+1}}{a_{k}}|)^{-1}$. By the assumption that $\frac{f_{\max}^2}{f_{\min}}<r$ and Lemma \ref{lem:TCBound}, we have 
 $$
 r>\frac{f_{\max}^2}{f_{\min}}> \|T_n\|_2>\max_j{|\lambda_j(T_n)|}>{|\lambda_j(T_n)|}\quad \text{for}\quad j=1,2,\dots,n,
 $$
 where $\lambda_j(T_n)$ denotes the $j$-th eigenvalue of $T_n$. By Theorem \ref{lem:matrixtaylorseries}, 
$
h(T_n) = \sum_{k=0}^{\infty}a_kT_n^{k}
$
is well-defined. Similarly,
$
h(A_n) = \sum_{k=0}^{\infty}a_kA_n^{k}.
$

We can  now decompose \[ h(T_n)-h(A_n) = \underbrace{h(T_n)-\sum_{k=0}^{K}a_kT_n^k}_{\Delta_n^{(1)}} + \underbrace{\sum_{k=0}^{K}a_kT_n^k - \sum_{k=0}^{K}a_kA_n^k}_{\Theta_n} + \underbrace{\sum_{k=0}^{K}a_kA_n^k -h(A_n)}_{\Delta_n^{(2)}}. \]

We first measure $\|\Delta_n^{(1)} + \Delta_n^{(2)}\|_2$. By Theorem \ref{lem:matrixtaylorserieserrorbd},
\begin{eqnarray}\nonumber
&&\|\Delta_n^{(1)} + \Delta_n^{(2)}\|_2\\\nonumber
&\leq&\|h(T_n)-\sum_{k=0}^{K}a_kT_n^k\|_2 + \|h(A_n)-\sum_{k=0}^{K}a_kA_n^k\|_2\\\nonumber
&\leq&\frac{\|T_n\|^{K+1} _2}{(K+1)!}\max_{0\leq t \leq 1 }\|h^{(k+1)}(tT_n) \|_2 + \frac{\|A_n\|^{K+1}_2}{(K+1)!} \max_{0\leq t \leq 1}\| h^{(k+1)}(tA_n) \|_2.
\end{eqnarray}

Before we provide a measure of $\|\Delta_n^{(1)} + \Delta_n^{(2)}\|_2,$ we notice by Lemma \ref{lem:TCBound} that
\begin{eqnarray}\nonumber
 \max_{0\leq t \leq 1}\| h^{(k+1)}(tT_n) \|_2&=& \max_{0\leq t \leq 1}\|\sum_{k=0}^{\infty}\frac{(K+k+1)!}{k!}a_{K+k+1}(tT_n)^k \|_2\\\nonumber
&\leq& \sum_{k=0}^{\infty}\frac{(K+k+1)!}{k!}|a_{K+k+1}|\|T_n\|_2^k \\\nonumber
&\leq& \sum_{k=0}^{\infty}\frac{(K+k+1)!}{k!}|a_{K+k+1}|(\frac{f_{\max}^2}{f_{\min}})^k.
\end{eqnarray}

It can be shown that the series $\sum_{k=0}^{\infty}\frac{(K+k+1)!}{k!}|a_{K+k+1}|(\frac{f_{\max}^2}{f_{\min}})^k$ is convergent. By the assumption $\frac{f_{\max}^2}{f_{\min}}<r=(\lim_{k\to\infty}|\frac{a_{k+1}}{a_{k}}|)^{-1}$, 
\[
\lim_{k\to\infty}|\frac{a_{K+k+2}}{a_{K+k+1}}| (\frac{K+k+2}{k+1})\frac{f_{\max}^2}{f_{\min}}=\lim_{k\to\infty}|\frac{a_{k+1}}{a_{k}}|\frac{f_{\max}^2}{f_{\min}}<(\frac{1}{r})r=1.
\]
By the ratio test, the series is convergent and is therefore independent of $n$. We then denote it by $\sum_{k=0}^{\infty}\frac{(K+k+1)!}{k!}|a_{K+k+1}|(\frac{f_{\max}^2}{f_{\min}})^k=:m_{(\frac{f_{\max}^2}{f_{\min}})}$.
	
%
Using a similar argument for $\max_{0\leq t \leq 1}\| h^{(k+1)}(tA_n) \|_2$, we have 
\begin{eqnarray}\nonumber
\|\Delta_n^{(1)} + \Delta_n^{(2)}\|_2
&\leq&\frac{\|T_n \|_2^{K+1}}{(K+1)!}m_{(\frac{{f_{\max}^2}}{f_{\min}})}+ \frac{\|A_n \|_2^{K+1}}{(K+1)!}m_{({f_{\max}})}\\\nonumber
&\leq&\frac{(\frac{{f_{\max}^2}}{f_{\min}})^{K+1}}{(K+1)!}m_{(\frac{{f_{\max}^2}}{f_{\min}})}+ \frac{f_{\max}^{K+1}}{(K+1)!}m_{({f_{\max}})}=:\epsilon_{K}
\end{eqnarray}
which tends to zero when $K$ is sufficiently large. Therefore, for a given $\epsilon_K>0$, there exists an integer $K$ such that
\begin{equation}
 \|\Delta_n^{(1)} + \Delta_n^{(2)}\|_2\leq\epsilon_{K}\leq \epsilon. \label{eqn:B1B2}
 \end{equation}

Next, we show that $\Theta_n$ can be further decomposed into the sum of a low rank matrix and a small norm matrix. By Theorem \ref{lem:mianTAdecompose}, for all $\epsilon>0$ there exist integers $N_1$ and $M_2>0$ such that for all $n>N_1$, we have
\[T_n  - A_n  = V_n + W_n,\]where
 $ V_{n}=\begin{bsmallmatrix}{}
 &  &  & \Diamond & \cdots &\Diamond \\
 &  &  &   & \ddots & \vdots \\
 & & & & & \Diamond  \\
\Diamond &  &  &  &  \\
\vdots & \ddots & & &   &  \\
\Diamond & \cdots& \Diamond  &&  & 
\end{bsmallmatrix}$ with rhombuses representing nonzero entries,
\[ \text{rank}(V_n) \leq 2M_1, \quad \text{and} \quad\|W_n\|_2 \leq \epsilon. \]
We then decompose $\Theta_n$ into
\begin{eqnarray}\nonumber
\Theta_n &=& \sum_{k=0}^{K}a_kT_n^k - \sum_{k=0}^{K}a_kA_n^k\\\nonumber
&=&\sum_{k=1}^{K}a_k( \sum_{j=0}^{k-1} T_n ^{j}( T_n -A_n )A_n ^{k-1-j} )\\\nonumber
&=&\sum_{k=1}^{K}a_k( \sum_{j=0}^{k-1} T_n ^{j}(V_n + W_n)A_n ^{k-1-j} )\\\nonumber
&=&\underbrace{ \sum_{k=1}^{K}a_k( \sum_{j=0}^{k-1} T_n ^{j}V_nA_n ^{k-1-j} )}_{R_n } + \underbrace{\sum_{k=1}^{K}a_k( \sum_{j=0}^{k-1} T_n ^{j}W_nA_n ^{k-1-j} )}_{\Delta_n^{(3)}}.
 \end{eqnarray}


By Lemma \ref{lem:TCBound},
\begin{eqnarray}\nonumber
\|\Delta_n^{(3)}\|_2 &=& \|\sum_{k=1}^{K}a_k \sum_{j=0}^{k-1} T_n ^{j}W_nA_n ^{k-1-j} \|_2\\\nonumber
&\leq &\|W_n\|_2\sum_{k=1}^{K}|a_k| \sum_{j=0}^{k-1}\| T_n \|_2^{j}\|A_n \|_2^{k-1-j}\\\label{eqn:Jnorm}
&\leq &\epsilon \underbrace{\sum_{k=1}^{K}|a_k| \sum_{j=0}^{k-1}\frac{f_{\max}^{k-1+j}}{f_{\min}^{j}}}_{m_0}
\end{eqnarray} where $m_0$ is a constant independent of $n$.

We now estimate the rank of $R_n$ by inspecting its sparsity structure. Using a computational lemma given in \cite{doi:10.1137/080720280} (see the proof of Lemma 3.11 wherein), simple calculations give
 $ 
  T_n ^{\alpha}V_nA_n ^{\beta}=\begin{bsmallmatrix}{}
\Diamond & \cdots & \Diamond &  & \Diamond & \cdots& \Diamond \\
\vdots  & \Diamond  & \vdots  &  &\vdots  & \Diamond& \vdots \\
\Diamond & \cdots  & \Diamond&  & \Diamond & \cdots& \Diamond \\
  &   & &    && &  \\
\Diamond & \cdots & \Diamond &  & \Diamond & \cdots& \Diamond \\
\vdots  & \Diamond  & \vdots  &  &\vdots  &\Diamond & \vdots\\
\Diamond & \cdots  & \Diamond&  & \Diamond & \cdots& \Diamond 
\end{bsmallmatrix},
$
where the rhombuses represent the nonzero entries. Assuming $n>2\max(\alpha+1, \beta+1)M_1$, these entries appear only in the four $(\alpha+1)M_1$ by $(\beta+1)M_1$ blocks located in the corners. As the rank of $R_n =\sum_{k=1}^{K}a_k( \sum_{j=0}^{k-1} T_n ^{j}V_nA_n ^{k-1-j} )$ is determined by that of $\sum_{j=0}^{K-1}T_n ^{j}V_nA_n ^{K-1-j}$, which is a matrix with only four nonzero $KM_1$ by $KM_1$ blocks in its corners, the rank of $R_n$ is bounded by $2KM_1$ provided that $n>2KM_1$.

Therefore, combining (\ref{eqn:B1B2}) and (\ref{eqn:Jnorm}), we let $N := \max{\{N_1, 2KM_1  \}} $ and conclude that for all $n>N$  

\[ h({T_n}) - h({A_n}) = R_n + \Delta_n^{(1)}+\Delta_n^{(2)}+\Delta_n^{(3)},\]
where $$\text{rank}(R_n) \leq 2\underbrace{KM_1}_{M},$$
and
$$\|\underbrace{\Delta_n^{(1)}+\Delta_n^{(2)}+\Delta_n^{(3)}}_{E_n}\|_2\leq(m_0+1)\epsilon_.$$\qed
\end{proof}

\begin{corollary}
\label{coro:mainTA_hon17}
Suppose $h(z)$ is an analytic function defined on $|z|<r$ with the radius of converge $r$. Let $f$ be a positive function in the Wiener class such that $\frac{f_{\max}^2}{f_{\min}}<r$. Let $A_n\in \mathbb{C}^{n \times n}$ be the Toeplitz matrix generated by $f$ and $T_n\in \mathbb{C}^{n \times n}$ be the superoptimal circulant preconditioner for $A_n$. If $\|h(T_n)^{-1}\|_2$ is uniformly bounded with respect to $n$, then for all $\epsilon > 0$ there exist positive integers $N$ and $M$ such that for all $n > N$ \[ |h({T_n})|^{-1}h({A_n}) = Q_n + \widetilde{R}_n + \widetilde{E}_n,\] where $Q_n$ is Hermitian and unitary, \[ \text{rank}(\widetilde{R}_n)\leq 2M,\quad \text{and} \quad \|\widetilde{E}_n\|_2 \leq \epsilon.\]
\end{corollary}
\begin{proof}

We rewrite $|h{(T_n)}|$ as
\begin{eqnarray}\nonumber
|h{(T_n)}| &=& U_n^* |h({\Lambda_n})|U_n \\\nonumber
&=&  U_n^* h({\Lambda_n})U_n \underbrace{U_n^* \overline{h({\Lambda_n})}U_n}_{Q_n}\\\label{eqn:abdC}
&=& h{(T_n)} Q_n,
\end{eqnarray}
where $\overline{h({\Lambda_n})}$ is the diagonal matrix containing the sign of the eigenvalues of $h(T_n)$. As $h(T_n)$ being Hermitian has only real-valued eigenvalues, $\overline{h({\Lambda_n})}$ is in fact is a diagonal matrix with $\pm 1$ in its diagonal. Therefore, $Q_n$ is both Hermitian and unitary. 

By Theorem \ref{thm:mainTAdecomp_hon17}, for all $\epsilon > 0$ there exist positive integers $N$ and $M$ such that for all $n > N$ \[ h(T_n) - h(A_n) = R_n  + E_n ,\] where \[ \text{rank}(R_n) \leq 2M  \quad \text{and} \quad  \|E_n \|_2 \leq \epsilon.\]

By the uniform boundedness assumption that $\|h(T_n)^{-1}\|_2 < c_0 $ for $n=1,2, \dots$, where $c_0$ is a positive constant independent of $n$, 
\begin{eqnarray}\nonumber
h(T_n)^{-1}h(A_n) &=& I_n + h(T_n)^{-1}(h(A_n)-h(T_n)) \\\nonumber
&=& I_n + h(T_n)^{-1}(-R_n )+{h(T_n)^{-1}(-E_n )}.
\end{eqnarray}

By (\ref{eqn:abdC}), we then obtain
\begin{eqnarray}\nonumber
|h(T_n)|^{-1}h(A_n) &=& Q_n h(T_n)^{-1}h(A_n)\\\nonumber
& =& Q_n + \underbrace{Q_nh(T_n)^{-1}(-R_n ) }_{\widetilde{{R}}_n } + \underbrace{Q_nh(T_n)^{-1}(-E_n ) }_{\widetilde{{E}}_n },
\end{eqnarray}
where
\[
\text{rank}(\widetilde{{R}}_n )= \text{rank}(Q_nh(T_n)^{-1}R_n) = \text{rank}({{R}}_n ) \leq 2M
\] and 
\[
\|\widetilde{{E}}_n \|_2 = \|Q_nh(T_n)^{-1}E_n \|_2 =  \|h(T_n)^{-1}E_n \|_2 \leq c_0\epsilon.
\]
\qed
\end{proof}

\begin{remark}
Note that the bound of $\|h(T_n)^{-1}\|_2$ depends on $h(z)$ as well as the generating function $f$. For example, considering $h(z)=\cos{z}$, we can have a case in which 
\[
\|(\cos{T_n})^{-1}\|_2=\max_j{|\frac{1}{\cos{\lambda_j}} |},
\]
where $\lambda_j$ is the $j$-th eigenvalue of $T_n$. Namely, $\|(\cos{T_n})^{-1}\|_2$ goes to infinity as $\cos{\lambda_j}$ approaches zero. Therefore, the uniform boundedness condition on $\|h(T_n)^{-1}\|_2$ is required.
\end{remark}

We can now show that the eigenvalues of $|h({T_n})|^{-1}h({A_n})$ are clustered around $\pm 1$ using Corollary \ref{coro:mainTA_hon17}. Note however that both $\widetilde{R}_n$ and $\widetilde{E}_n$ in the corollary are not Hermitian in general. Besides, one must deal with the unitary matrix $Q_n$ instead of the usual identity matrix in the matrix decomposition. Therefore, Cauchy's interlace theorem that was used for example in \cite{MR2376196} to show clustered spectra does not straightforwardly apply. Nevertheless, we are still able to show the clustered spectra of our concerned preconditioned matrix via a simple trick.

\begin{corollary}
\label{coro:mainTA_clusters_hon17}
Suppose $h(z)$ is an analytic function defined on $|z|<r$ with the radius of converge $r$. Let $f$ be a positive function in the Wiener class such that $\frac{f_{\max}^2}{f_{\min}}<r$. Let $A_n\in \mathbb{C}^{n \times n}$ be the Toeplitz matrix generated by $f$ and $T_n\in \mathbb{C}^{n \times n}$ be the superoptimal circulant preconditioner for $A_n$. If $\|h(T_n)^{-1}\|_2$ is uniformly bounded with respect to $n$, then $|h({T_n})|^{-1}h({A_n})$ has clustered spectra around $\pm 1$ for sufficiently large $n$.
\end{corollary}

\begin{proof}
By Corollary \ref{coro:mainTA_hon17}, for all $\epsilon > 0$ there exist positive integers $N$ and $M$ such that for all $n > N$ 
\begin{eqnarray}\nonumber
&&\underbrace{|h({T_n})|^{-\frac{1}{2}}h({A_n})|h({T_n})|^{-\frac{1}{2}}}_{H_n} \\\nonumber
&=& \underbrace{|h({T_n})|^{\frac{1}{2}}Q_n|h({T_n})|^{-\frac{1}{2}}}_{\overline{Q}_n} + \underbrace{|h({T_n})|^{\frac{1}{2}}\widetilde{R}_n|h({T_n})|^{-\frac{1}{2}}}_{\overline{R}_n} + \underbrace{|h({T_n})|^{\frac{1}{2}}\widetilde{E}_n|h({T_n})|^{-\frac{1}{2}}}_{\overline{E}_n},
\end{eqnarray}
where $\overline{Q}_n$ is unitary and is similar to $Q_n$, \[ \text{rank}(\overline{R}_n)\leq 2M, \quad \text{and} \quad\|\overline{E}_n\|_2 \leq \epsilon,\] provided that $\|h(T_n)^{-1}\|_2$ is uniformly bounded with respect to $n$.

We introduce the following matrix decomposition

\[
\underbrace{\begin{bmatrix}{}
 & H_n \\
H_n^* & \\
\end{bmatrix}}_{\mathcal{H}}=\underbrace{\begin{bmatrix}{}
 & \overline{Q}_n \\
\overline{Q}_n^* & \\
\end{bmatrix}}_{\mathcal{Q}}+\underbrace{\begin{bmatrix}{}
 & \overline{R}_n \\
\overline{R}_n^* & \\
\end{bmatrix}}_{\mathcal{R}}+\underbrace{\begin{bmatrix}{}
 & \overline{E}_n \\
\overline{E}_n^* & \\
\end{bmatrix}}_{\mathcal{E}},
\]
where $\mathcal{Q}$ is unitary,
\[
\text{rank}(\mathcal{R})\leq 4M, \quad \text{and} \quad \text{rank}(\mathcal{E})\leq 2\epsilon.
\]

Note that all $\mathcal{H}$, $\mathcal{Q}$, $\mathcal{R}$, and $\mathcal{E}$ are Hermitian. By Corollary $3$ in \cite{BRANDTS20103100}, we know that there are at most $2(4M)=8M$ eigenvalues of $\mathcal{H}$ that are not around $\pm 1$. Thus, $\mathcal{H}$ has clustered spectra around $\pm 1$ by Definition \ref{def:clsutered_spec}. As the eigenvalues of $\mathcal{H}$ are the same as the singular values of $H_n$ up to $\pm$ sign, the singular values of $H_n$ are clustered around $1$. Consequently, as $H_n$ is Hermitian and is similar to $|h(T_n)|^{-1}h(A_n)$, we conclude that $|h(T_n)|^{-1}h(A_n)$ has clustered spectra around $\pm 1$.
\qed
\end{proof}


Using Lemma  \ref{lem:TCBound} and Theorem \ref{lem:mianSAdecompose}, we can show similar results for Strang's circulant preconditioners.

\begin{theorem}
\label{thm:mainSAdecomp_hon17}
Suppose $h(z)$ is an analytic function defined on $|z|<r$ with the radius of converge $r$. Let $f$ be a positive function in the Wiener class such that $f_{\max}<r$. Let $A_n \in \mathbb{C}^{n \times n}$ be the Toeplitz matrix generated by $f$ and $S_n \in \mathbb{C}^{n \times n}$ be Strang's circulant preconditioner for $A_n$. Then, for all $\epsilon > 0$ there exist positive integers $N$ and $M$ such that for all $n > N$ \[ h({S_n}) - h({A_n}) = R_n + E_n,\] where \[ \text{rank}(R_n) \leq 2M \quad \text{and} \quad \|E_n \|_2 \leq \epsilon.\]
\end{theorem}

\begin{corollary}
\label{coro:mainST_hon17}
Suppose $h(z)$ is an analytic function defined on $|z|<r$ with the radius of converge $r$. Let $f$ be a positive function in the Wiener class such that $f_{\max}<r$. Let $A_n \in \mathbb{C}^{n \times n}$ be the Toeplitz matrix generated by $f$ and $S_n \in \mathbb{C}^{n \times n}$ be Strang's circulant preconditioner for $A_n$. If $\|h(S_n)^{-1}\|_2$ is uniformly bounded with respect to $n$, then for all $\epsilon > 0$ there exist positive integers $N$ and $M$ such that for all $n > N$ \[ |h({S_n})|^{-1}h({A_n}) = Q_n + \widetilde{R}_n + \widetilde{E}_n,\] where $Q_n$ is Hermitian and unitary, \[ \text{rank}(\widetilde{R}_n)\leq 2M, \quad \text{and} \quad \|\widetilde{E}_n\|_2 \leq \epsilon.\]
\end{corollary}

\begin{corollary}
\label{coro:mainST_clusters_hon17}
Suppose $h(z)$ is an analytic function defined on $|z|<r$ with the radius of converge $r$. Let $f$ be a positive function in the Wiener class such that $f_{\max}<r$. Let $A_n\in \mathbb{C}^{n \times n}$ be the Toeplitz matrix generated by $f$ and $S_n\in \mathbb{C}^{n \times n}$ be Strang's circulant preconditioner for $A_n$. If $\|h(S_n)^{-1}\|_2$ is uniformly bounded with respect to $n$, then $|h({S_n})|^{-1}h({A_n})$ has clustered spectra around $\pm 1$ for sufficiently large $n$.
\end{corollary}


Note that $h(A_n)$ is Hermitian when $A_n$ is Hermitian. Similarly, $h(T_n)$ (of $h(S_n)$) is Hermitian when $T_n$ (or $S_n$) is Hermitian. Hence, we consider the following cases: (i) when $h(A_n)$ is Hermitian indefinite, MINRES can be used with $|h(T_n)|$ as a preconditioner. (ii) In the special case in which $h(A_n)$ is Hermitian positive definite, CG with $|h(T_n)|$ can then be employed.

\section{Extension to block Toeplitz matrices with Toeplitz blocks}

Our results on Toeplitz matrices given in \cite{Hon2018148} can be extended to block Toeplitz matrices with Toeplitz blocks (BTTB). In this section, we provide several theorems that show the effectiveness of the optimal block circulant preconditioners with circulant blocks (BCCB) for functions of BTTB matrices.


A BTTB matrix $A_{(n,m)} \in \mathbb{C}^{nm \times nm}$ is given by
\[ A_{(n,m)}=\begin{bmatrix}
A_{(0)} & A_{(-1)} & \cdots & A_{(-(n-1))} \\
A_{(1)} & \ddots & \ddots & \vdots \\
\vdots & \ddots & \ddots  & A_{(-1)} \\
A_{(n-1)} & \cdots & A_{(1)} & A_{(0) }
\end{bmatrix}  \] where the blocks $A_{(k)} \in \mathbb{C}^{m \times m}$, $|k|\leq n-1$, are Toeplitz matrices. We denote the entries of $A_{(n,m)}$ by $[A_{(n,m)}]_{p,q:r,s}=a_{p-q}^{(r-s)}$ for $1\leq r,s \leq n$ and $1\leq p,q \leq m$. Like Toeplitz matrices, we assume that $A_{mn}$ is associated with a generating function $f(x,y)$ defined on $[-\pi,\pi]\times[-\pi,\pi]$ and its Fourier coefficients are given by
\[ a_{k}^{(j)}=\frac{1}{(2\pi)^2} \int_{-\pi} ^{\pi} \int_{-\pi} ^{\pi}f(x,y) e^{-\mathbf{i}(jx+ky)} \,dxdy,\quad j,k=0,\pm1,\pm2,\dots.  \]

Throughout, we assume that $f$ is in the Wiener class, i.e. $$\sum_{j=-\infty}^{\infty}\sum_{k=-\infty}^{\infty}|a_k^{(j)}|<\infty.$$ The corresponding matrix $A_{mn}$ is therefore Hermitian for all $n$ and $m$. Again, we refer to \cite{MR2376196, MR2108963} for more about BTTB matrices.

We then introduce the absolute value BCCB matrices. Note that BCCB matrices are diagonalisable by the 2-dimensional Fourier matrix $ U_{n}\otimes U_{m}$.

\begin{definition}\label{def:BCCBmatrix}
Let ${C}_{(n,m)} \in \mathbb{C}^{n m\times n m}$ be a block circulant matrix with circulant blocks (BCCB). The \emph{absolute value BCCB matrix} $|{C}_{(n,m)}| \in \mathbb{C}^{n m \times n m} $ for ${C}_{(n,m)}$ is defined by
\begin{eqnarray}\nonumber
|{C}_{(n,m)}|&=&({C}^*_{(n,m)}{C}_{(n,m)})^{1/2}\\\nonumber
&=&({C}_{(n,m)}{C}^*_{(n,m)})^{1/2}\\\nonumber
&=&(U_{n}\otimes U_{m})^*
|\Omega_{(n,m)}|
 (U_{n}\otimes U_{m}),
\end{eqnarray}
where $U_n \in \mathbb{C}^{n \times n}$ is the Fourier matrix and $|\Omega_{(n,m)}| \in  \mathbb{R}^{n m \times n m}$ is the diagonal matrix in the eigendecomposition of ${C}_{(n,m)}$ with all entries replaced by their magnitude.
\end{definition}

\begin{remark}
Note that $|{C}_{(n,m)}|$ is Hermitian positive definite by definition provided that ${C}_{(n,m)}$ is nonsingular.
\end{remark}

Given an analytic function $h$, $|h({C}_{(n,m)})|=U_n^*|h(\Omega_{(n,m)})|U_n$ is well-defined and is a BCCB matrix by Definition \ref{def:BCCBmatrix}. Therefore, for any vector $\mathbf{d}$ the product $|h({C}_{(n,m)})|^{-1}\mathbf{d}$ can be efficiently computed by 2-dimensional FFTs in $\mathcal{O}(nm\log{nm})$ operations (see Section 5.2.2 in \cite{MR2376196}).

The following results can be shown using the similar arguments given in the previous section and we therefore omit their proofs.

\begin{theorem}
\label{thm:mainTAdecomp_hon17_BTTB}
Suppose $h(z)$ is an analytic function defined on $|z|<r$ with the radius of converge $r$. Let $f$ be a function in the Wiener class such that $|f|_{\max}<r$. Let ${A}_{(n,m)} \in \mathbb{C}^{n m \times n m} $ be the BTTB matrix generated by $f$ and $c({A}_{(n,m)}) \in \mathbb{C}^{n m \times n m} $ be the optimal BCCB preconditioner for ${A}_{(n,m)}$. Then, for all $\epsilon > 0$ there exist positive integers $N$ and $M$ such that for all $n > N$ and all $m > M$ \[ h(c({A}_{(n,m)})) - h({{A}_{(n,m)}}) = {R}_{(n,m)} + {E}_{(n,m)},\] where \[ \text{rank}({R}_{(n,m)}) \leq \mathcal{O}(n)+\mathcal{O}(m) \quad \text{and} \quad \|{E}_{(n,m)} \|_2 \leq \epsilon.\]
\end{theorem}

\begin{corollary}
\label{coro:mainTT_hon17_BTTB}
Suppose $h(z)$ is an analytic function defined on $|z|<r$ with the radius of converge $r$. Let $f$ be a function in the Wiener class such that $|f|_{\max}<r$. Let ${A}_{(n,m)} \in \mathbb{C}^{n m \times n m} $ be the BTTB matrix generated by $f$ and $c({A}_{(n,m)}) \in \mathbb{C}^{n m \times n m} $ be the optimal BCCB preconditioner for ${A}_{(n,m)}$. If $\|h(c({A}_{(n,m)}))^{-1}\|_2$ is uniformly bounded with respect to $n$ and $m$, then for all $\epsilon > 0$ there exist positive integers $N$ and $M$ such that for all $n > N$ and all $m > M$ \[ |h({c({A}_{(n,m)}})|^{-1}h({A}_{(n,m)}) = {Q}_{(n,m)} + \widetilde{R}_{(n,m)} + \widetilde{E}_{(n,m)},\] where $Q_{(n,m)}$ is Hermitian and unitary, \[ \text{rank}(\widetilde{R}_{(n,m)})\leq \mathcal{O}(n)+\mathcal{O}(m), \quad \text{and} \quad \|\widetilde{E}_{(n,m)}\|_2 \leq \epsilon.\]
\end{corollary}

\begin{corollary}
\label{coro:mainTT_clusters_BTTB}
Suppose $h(z)$ is an analytic function defined on $|z|<r$ with the radius of converge $r$. Let $f$ be a function in the Wiener class such that $|f|_{\max}<r$. Let ${A}_{(n,m)} \in \mathbb{C}^{n m \times n m} $ be the BTTB matrix generated by $f$ and $c({A}_{(n,m)}) \in \mathbb{C}^{n m \times n m} $ be the optimal BCCB preconditioner for ${A}_{(n,m)}$. If $\|h(c({A}_{(n,m)}))^{-1}\|_2$ is uniformly bounded with respect to $n$ and $m$, then for all $\epsilon > 0$ there exist positive integers $N$ and $M$ such that for all $n > N$ and all $m > M$ there are at most $\mathcal{O}(n)+\mathcal{O}(m)$ eigenvalues of $|h({c({A}_{(n,m)})})|^{-1}h({{A}_{(n,m)}})$ have absolute value larger than $\epsilon$.
\end{corollary}

We remark that our results could be further  generalised to functions of multilevel Toeplitz systems. However, negative results were in fact shown by Serra-Capizzano and Tyrtyshnikov \cite{doi:10.1137/S0895479897331941} that circulant type preconditioners are not optimal for multilevel Toeplitz systems in the sense that the spectra of the preconditioned matrices are not tightly clustered, i.e. the number of eigenvalues away $\pm 1$ grow with the dimensions.

\section{Numerical results}

In this section, we demonstrate the effectiveness of our proposed preconditioners using CG, MINRES, and GMRES. Throughout all numerical tests, $e^{A_n}$ is computed by the MATLAB R2016b built-in function \textbf{expm} while the other matrix functions are computed by \textbf{funm}. 
Also, we use the function \textbf{pcg} (or \textbf{minres}) to solve the Hermitian positive definite (or indefinite) system $ h(A_n )\mathbf{x}=\mathbf{b}, $
where $\mathbf{b}$ is generated by the function \textbf{ones(n,1)}, with the zero vector as the initial guess. As a comparison, GMRES is used and is executed by \textbf{gmres}. The stopping criterion used is $$ \frac{\|\mathbf{r}^{(j)}\|_2}{\|\mathbf{b}\|_2} < 10^{-7},$$ where $\mathbf{r}^{(j)}$ is the residual vector after $j$ iterations.

For all tests, the Toeplitz matrix $A_n$ is generated by $$f(
x)=2\sum_{k=0}^{\infty}\frac{\sin{(kx)}+\cos{(kx)}}{(1+k)^{1.1}}$$ in the Wiener class analysed in \cite{MR1098422} unless mentioned otherwise.

\textbf{Example 5.1}. We first consider the Toeplitz matrix exponential $e^{A_n}$. Table \ref{tab:ExpACG_HON17} shows the numbers of iterations needed for convergence using CG and GMRES. It appears that the proposed precondtioners are effective for speeding up the convergence rate. Figure \ref{fig:ExpACG_HON17}  shows the spectrum of the matrix before or after the preconditioner $|e^{T_n}|$ is used at $n=512$. In Figure \ref{fig:ExpACG_HON17} (iii), we observe the highly clustered spectrum around $1$. By \cite{Axelsson1986} for example, a fast convergence rate of CG is expected due to the clustered eigenvalues.

\begin{table}
\caption{Numbers of iterations with (a) CG or (b) GMRES for $e^{A_n}$ given in Example 5.1.}
\label{tab:ExpACG_HON17}
\begin{center}
{
(a)~\begin{tabular}{|l|c|c|c|}\hline
 $n $ & with no preconditioner  & with $|e^{T_n}|$ & with $|e^{S_n}|$
\\
\hline
128 & 34 & 11 & 11\\
256 & 53 & 11 & 11\\
512 & 79 & 11 & 12\\
1024 & 121 & 12 & 13\\
\hline
\end{tabular}\\
\vspace{0.5cm}
(b)~\begin{tabular}{|l|c|c|c|}\hline
 $n $ & with no preconditioner  & with $e^{T_n}$ & with $e^{S_n}$
\\
\hline
128 & 26 & 11 & 11\\
256 & 35 & 11 & 12\\
512 & 46 & 12 & 13\\
1024 & 62 & 12 & 13\\
\hline
\end{tabular}}
\end{center}
\end{table}

\begin{figure}
\centering
(i)~\includegraphics[scale=1.0]{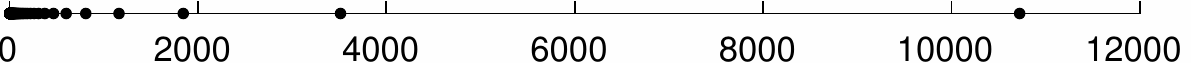}\\
\vspace{0.5cm}
(ii)~\includegraphics[scale=1.0]{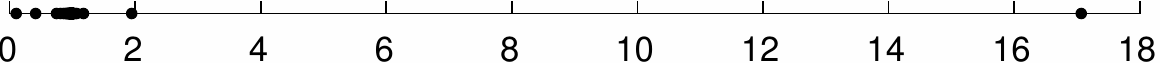}\\
\vspace{0.5cm}
(iii)~\includegraphics[scale=1.0]{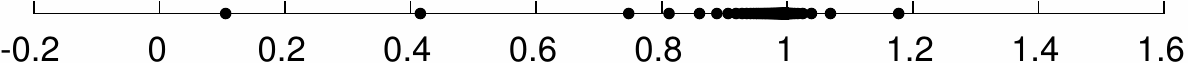}
\caption{Spectrum of $e^{A_n}$ given in Example 5.1 at $n=512$ (i) without a preconditioner or (ii) with the preconditioner $|e^{T_n}|$. (iii) Zoom-in spectrum of (ii).}
\label{fig:ExpACG_HON17}
\end{figure}

\newpage

\textbf{Example 5.2}. Table \ref{tab:cosA_HON17} shows the numerical results using MINRES and GMRES for $\cos{A_n}$. Again, we observe that the preconditioners appear effective for $\cos{A_n}$. In Figure \ref{fig:cosA_HON17}, we further show the spectra of $|\cos{T_n}|^{-1}\cos{A_n}$ at different $n$. We conclude that the highly clusters of eigenvalues around $\pm1$ seem independent of $n$.

\begin{table}[h]
\caption{Numbers of iterations with (a) MINRES or (b) GMRES for $\cos{A_n}$ given in Example 5.2.}
\label{tab:cosA_HON17}
\begin{center}
{
(a)~\begin{tabular}{|l|c|c|c|}\hline
 $n $ & with no preconditioner  & with $|\cos{T_n}|$ & with $|\cos{S_n}|$

\\
\hline
128 & 178 & 29 & 42\\
256 & 412 & 32 & 50\\
512 & 952 & 49 & 46\\
1024 & 2152 & 47 & 48\\
\hline
\end{tabular}\\
\vspace{0.5cm}
(b)~\begin{tabular}{|l|c|c|c|}\hline
 $n $ & with no preconditioner  & with $\cos{T_n}$ & with $\cos{S_n}$

\\
\hline
128 & 128 & 18 & 21\\
256 & 256 & 18 & 20\\
512 & 512 & 21 & 24\\
1024 & 1024 & 21 & 24\\
\hline
\end{tabular}}\\
\end{center}
\end{table}

\begin{figure}
\centering
(i)~\includegraphics[scale=1.0]{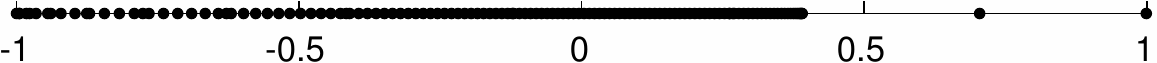}\\
\vspace{0.5cm}
(ii)~\includegraphics[scale=1.0]{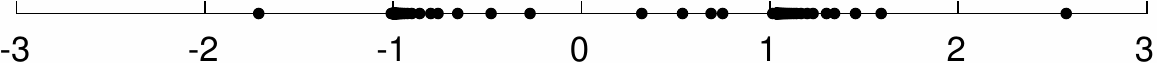}\\
(a) $n=128$\\
\vspace{2cm}
(i)~\includegraphics[scale=1.0]{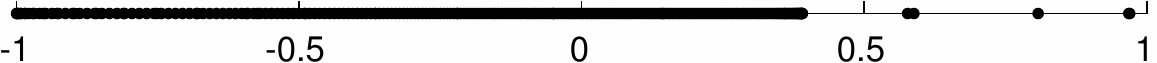}\\
\vspace{0.5cm}
(ii)~\includegraphics[scale=1.0]{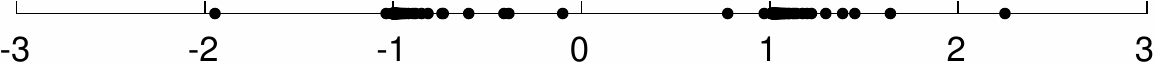}\\
(b) $n=256$\\
\vspace{2cm}
(i)~\includegraphics[scale=1.0]{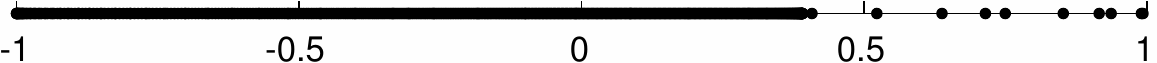}\\
\vspace{0.5cm}
(ii)~\includegraphics[scale=1.0]{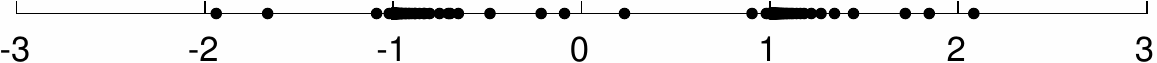}\\
(c) $n=512$\\

\caption{Spectra of $\cos{A_n}$ given in Example 5.2 at different $n$ (i) without a preconditioner or (ii) with the preconditioner $|\cos{T_n}|$. }
\label{fig:cosA_HON17}
\end{figure}

\newpage

\textbf{Example 5.3}. We next consider the system defined by the hyperbolic sine function. Table \ref{tab:sinhA_HON17} shows the numerical results using CG and GMRES for $\sinh{A_n}$. The numbers of iterations needed for convergence are reduced significantly with the proposed preconditioners. In Figure \ref{fig:SsinhACG_HON17}, we observe the cluster around $1$ at $n=512$ when the matrix is preconditioned by $|\sinh{S_n}|$.

\begin{table}[h]
\caption{Numbers of iterations with (a) CG or (b) GMRES for $\sinh{A_n}$ given in Example 5.3.}
\label{tab:sinhA_HON17}
\begin{center}
(a)~\begin{tabular}{|l|c|c|c|}\hline
 $n $ & with no preconditioner  & with $|\sinh{T_n}|$ & with $|\sinh{S_n}|$ 
\\
\hline
128 & 38 & 11 & 11 \\
256 & 57 & 11 & 12\\
512 & 82 & 11 & 12\\
1024 & 129 & 12 & 13 \\
\hline
\end{tabular}\\
\vspace{0.5cm}
(b)~\begin{tabular}{|l|c|c|c|}\hline
 $n $ & with no preconditioner & with $\sinh{T_n}$ & with $\sinh{S_n}$ 
\\
\hline
128 & 27 & 11 & 11 \\
256 & 36 & 11 & 12\\
512 & 47 & 12 & 13\\
1024 & 63 & 12 & 13 \\
\hline
\end{tabular}\\
\vspace{0.5cm}
\end{center}
\end{table}

\begin{figure}[h]
\centering
(i)~\includegraphics[scale=1.0]{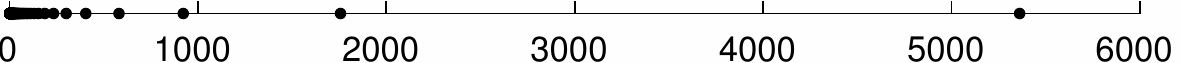}\\
\vspace{0.5cm}
(ii)~\includegraphics[scale=1.0]{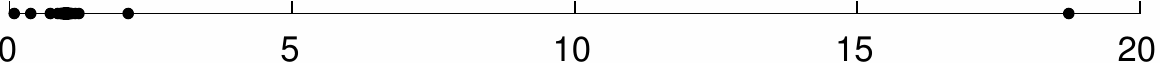}\\
\vspace{0.5cm}
(iii)~\includegraphics[scale=1.0]{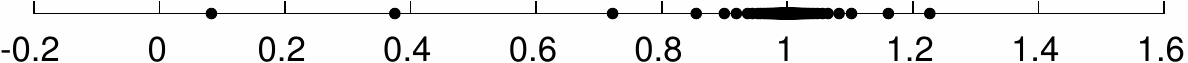}
\caption{Spectrum of $\sinh{A_n}$ given in Example 5.3 at $n=512$ (i) without a preconditioner or (ii) with the preconditioner $|\sinh{S_n}|$. (iii) Zoom-in spectrum of (ii).}
\label{fig:SsinhACG_HON17}
\end{figure}

\newpage

\textbf{Example 5.4}. In this example, we consider the polynomial $h(z)=z^3+z^2+z+1$. Table \ref{tab:polyhA_HON17} shows the numerical results for $h(A_n)$. The numbers of iterations needed for convergence are reduced with the preconditioners $|h(T_n)|$ and $|h(S_n)|$. In Figure \ref{fig:SpolyA_HON17}, we again observe a cluster of eigenvalues around $1$ when the system is preconditioned by $|h(S_n)|$.

\begin{table}[h]
\caption{Numbers of iterations with (a) CG or (b) GMRES for $h(A_n)$ given in Example 5.4.}
\label{tab:polyhA_HON17}
\begin{center}
{
(a)~\begin{tabular}{|l|c|c|c|}\hline
 $n $ & with no preconditioner & with $|h(T_n)|$ & with $|h(S_n)|$ 
\\
\hline
128 & 32 & 9 & 9 \\
256 & 40 & 9 & 9 \\
512 & 50 & 9 & 9 \\
1024 & 63 & 9 & 9 \\
\hline
\end{tabular}\\
\vspace{0.5cm}
(b)~\begin{tabular}{|l|c|c|c|}\hline
 $n $ & with no preconditioner  & with $h(T_n)$ & with $h(S_n)$ 
\\
\hline
128 & 27 & 10 & 10 \\
256 & 35 & 10 & 10 \\
512 & 43 & 9 & 10 \\
1024 & 51 & 10 & 10 \\
\hline
\end{tabular}}
\end{center}
\end{table}

\begin{figure}[h]
\centering
(i)~\includegraphics[scale=1.0]{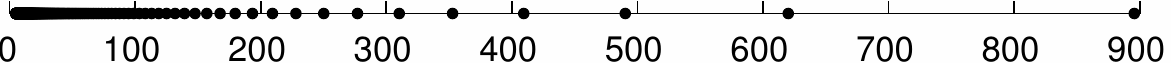}\\
\vspace{0.5cm}
(ii)~\includegraphics[scale=1.0]{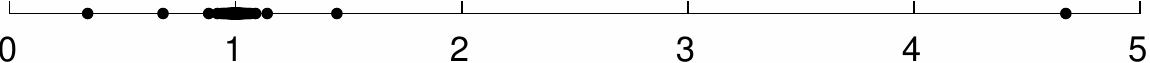}\\
\vspace{0.5cm}
\caption{Spectrum of $h{(A_n)}$ given in Example 5.4 at $n=512$ (i) without a preconditioner or (ii) with the preconditioner $|h{(S_n)}|$.}
\label{fig:SpolyA_HON17}
\end{figure}

\textbf{Example 5.5}. Lastly, we consider the BTTB matrix $A_{(n,m)}$ denoted by
\[
a^{(j)}_{k}=\frac{1}{(|j|+1)^{2.1}+ (|k|+1)^{2.1}}, \quad j,k=0,\pm 1, \pm 2, \dots,
\]
analysed in \cite{MR2376196}. The matrix function in this case is $e^{z}$. 

Table \ref{tab:example_5} illustrates the numerical results with our proposed preconditioners. Even though we observe reduction in iteration counts for both MINRES and GMRES, the eigenvalues of the preconditioned matrix are not tightly clustered as shown in Figure \ref{fig:example_5} due to the negative result mentioned in the previous chapter.

\begin{table}[h]
\caption{Numbers of iterations with (a) MINRES or (b) GMRES for $e^{A_{(n,m)}}$ given in Example 5.5.}
\label{tab:example_5}
\begin{center}
{
(a)~\begin{tabular}{|l|c|c|c|c|}\hline
 $n$ & $m$ & $nm$ & with no preconditioner   & with $|e^{c(A_{(n,m)})}|$

\\
\hline
16 & 8 & 128 & 12 & 8\\
16 & 16 & 256 & 13 & 8\\
32 & 16 & 512 & 22 & 12\\
32 & 32 & 1024& 26 & 11\\
\hline
\end{tabular}\\
\vspace{0.5cm}
(b)~\begin{tabular}{|l|c|c|c|c|}\hline
 $n$ & $m$ & $nm$ & with no preconditioner   & with $e^{c(A_{(n,m)})}$

\\
\hline
16 & 8 & 128 & 11 & 9\\
16 & 16 & 256 & 12 & 9\\
32 & 16 & 512 & 19 & 15\\
32 & 32 & 1024& 23 & 14\\
\hline
\end{tabular}}\\
\end{center}
\end{table}

\begin{figure}[h]
\centering
(i)~\includegraphics[scale=1.0]{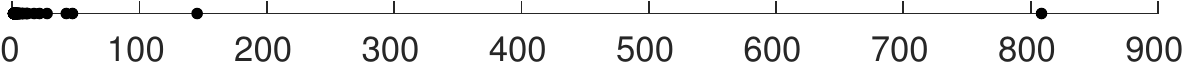}\\
\vspace{0.5cm}
(ii)~\includegraphics[scale=1.0]{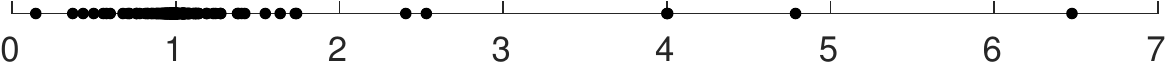}\\

\caption{Spectrum of $e^{A_{(n,m)}}$ given in Example 5.5 at $(n,m)=(32,16)$ (i) without a preconditioner or (ii) with the preconditioner $|e^{c(A_{(n,m)})}|$.} 
\label{fig:example_5}
\end{figure}

\section{Conclusions}

We have proposed the absolute value superoptimal circulant preconditioners $|h(T_n)|$ for analytic functions of Toeplitz matrices $h(A_n)$, where $A_n$ is generated by a positive function in the Wiener class. Also, we have provided several theorems that explain the effectiveness of such preconditioners. Specifically, we have shown that $|h(T_n)|^{-1}h(A_n)$ has clustered spectra around $\pm 1$, even though both the low rank matrix and the small norm matrix in the related matrix decomposition are not Hermitian. Furthermore, we have shown that the absolute value optimal BCCB preconditioners $|h(c(A_{(n,m)}))|$ are effective for functions of BTTB systems $h(A_{(n,m)})$, where $A_{(n,m)}$ is generated by a function in the Wiener class.

A number of numerical examples concerning different $h$ have been provided. In each of the examples, we observe a significant improvement for convergence and the expected clusters of eigenvalues with our proposed preconditioners. Although in some example GMRES with $h(T_n)$ (or $h(S_n)$) requires fewer iterations than MINRES with $|h(T_n)|$ (or $|h(S_n)|$), it is not necessarily reduction in work as the cost of GMRES increases for every iteration whereas MINRES has a constant cost per iteration. 

Our results show that the superoptimal circulant preconditioners as well as the optimal circulant preconditioners are effective for at least two special classes of matrices: functions of Toeplitz matrices and those of BTTB matrices. These matrices as examples extend the applicability of optimal type preconditioners in  the context of preconditioning for functions of matrices addressed by the authors in \cite{Jin2014224,bai2015}. It would not be inconceivable for such preconditioners to also work well for other Toeplitz-related systems. However, as for other more general matrices, the success of such circulant preconditioners is less obvious.




\section*{References}
\bibliographystyle{plain}
\bibliography{SeanReferences}

\end{document}